\documentclass[a4paper, 11pt]{article}
\usepackage{fullpage}

\usepackage[latin1]{inputenc}
\usepackage{graphicx}
\usepackage{amsfonts}
\usepackage{url} 
\usepackage{color}
\usepackage{amsthm} 
\usepackage[lined,boxruled,rightnl]{algorithm2e}
\usepackage{euscript}
\usepackage{multirow}
\usepackage{cite}

\newtheorem{assumption}{Assumption} 
\newtheorem{proposition}{Proposition} 
 
\newtheorem{lemma}{Lemma} 
\newtheorem{remark}{Remark}

\input alphabet
\input jidef

\newcommand{\minimize}[2]{\ensuremath{\underset{\substack{{#1}}}%
{\mathrm{minimize}}\;\;#2 }}

\newcommand{\ran}{\ensuremath{\operatorname{ran}}}

\newcommand{\as}{\ensuremath{\text{\rm $\PP$-a.s.}}}
\newcommand{\EC}[2]{{\mathsf E}(#1\! \mid\! #2)} 
\newcommand{\Argmind}[2]{\ensuremath{\underset{\substack{{#1}}}%
{\mathrm{Argmin}}\;\;#2 }}

\renewcommand{\leq}{\ensuremath{\leqslant}}
\renewcommand{\geq}{\ensuremath{\geqslant}}
\renewcommand{\le}{\ensuremath{\leqslant}}
\renewcommand{\ge}{\ensuremath{\geqslant}}

\newcommand{\FF}{\ensuremath{{\EuScript F}}}
\newcommand{\PP}{\ensuremath{\mathsf P}}
\newcommand{\EE}{\ensuremath{\mathsf E}}
\newcommand{\Span}{\ensuremath{\operatorname{span}}}
\renewcommand{\cv}{\mathbf{c}}

\newcommand{\XX}{\ensuremath{\boldsymbol{\EuScript{X}}}}

\newcommand {\newtext}[1]{\textcolor{black}{#1}}
\newcommand {\correcEC}[1]{\textcolor{black}{#1}}

\title{A Stochastic Majorize-Minimize Subspace Algorithm for Online Penalized Least Squares Estimation}
\author{Emilie Chouzenoux and Jean-Christophe Pesquet
\thanks{E. Chouzenoux (corresponding author) is with the Laboratoire d'Informatique Gaspard 
Monge, UMR CNRS 8049,
Universit\'e Paris-Est, 77454 Marne la Vall\'ee 
Cedex 2, France. E-mail: \texttt{emilie.chouzenoux@univ-paris-est.fr}.
J.-C. Pesquet is with the Center for Visual Computing, CentraleSupelec, University Paris-Saclay,
92295 Chatenay-Malabry, France. E-mail: \texttt{jean-christophe@pesquet.eu}. 
This work was supported by the CNRS Imag'in project under grant 2015 OPTIMISME. 
Part of it was presented at the EUSIPCO 2014 conference \cite{Chouzenoux14eusipco}.}
}

\begin{document}

\maketitle

\begin{abstract}
Stochastic approximation techniques play an important role in solving many problems encountered in machine learning or adaptive signal processing. 
In these contexts, the statistics of the data are often unknown a priori or their direct computation is too intensive, and they have thus to be estimated online from the observed signals.
For batch optimization of an objective function being the sum of a data fidelity term and a penalization 
(e.g. a sparsity promoting function), Majorize-Minimize (MM) methods have recently attracted much interest since they are fast, highly flexible, and effective in ensuring convergence. The goal of this paper is to show how these methods can be successfully extended to the case when the data fidelity term corresponds to a least squares criterion 
and the cost function is replaced by a sequence of stochastic approximations of it. In this context, we propose an online version of an MM subspace algorithm and we study its convergence by using suitable probabilistic tools.
Simulation results illustrate the good practical performance of the proposed algorithm associated with a memory gradient subspace,
when applied to both non-adaptive and adaptive filter identification problems.
\end{abstract}

%
{\small \emph{Keywords}: stochastic approximation, optimization, subspace algorithms, memory gradient methods, descent methods, recursive algorithms, majorization-minimization, filter identification, Newton method, sparsity, machine learning, adaptive filtering.}
%

\section{Introduction}
A classical problem in data sciences consists of inferring the structure of a linear model linking some observed
random variables $(\Xv_n)_{n\ge 1}$ in $\eR^{N\times Q}$ to some other observed random variables $(\yv_n)_{n\ge 1}$ in $\eR^Q$.
Unless otherwise specified, we will assume in this work that the following wide-sense stationarity properties hold:
\begin{align}
(\forall n \in \correcEC{\eN \setminus \left\{0\right\}})\qquad & \EE(\|\yv_n\|^2) = \varrho \\
&\EE(\Xv_n \yv_n ) = \rb \label{e:statrb}\\
&\EE(\Xv_n \Xv_n^\top ) = \Rb, \label{e:statRb}
\end{align}
where $\varrho\in (0,+\infty)$, $\rb \in \eR^N$, $\Rb\in \eR^{N\times N}$ is a symmetric positive semi-definite matrix, \newtext{$\EE(\cdot)$ denotes the mathematical expectation, and $\|\cdot\|$ is the Euclidean norm.}
We will then be interested in the following optimization formulation:
\begin{equation}\label{e:probdet}
\minimize{\hb\in \eR^N}{F(\hb)},
\end{equation} 
\correcEC{with\footnote{The wide sense stationarity assumption makes $F$ independent of the choice of $n \in \eN \setminus \left\{0\right\}$ .}} 
\begin{equation}
(\forall \hb\in \eR^N)\quad
 F(\hb) =\frac12 \EE\big(\|\yv_n-\Xv_n^\top \hb\|^2\big)+ \Psi(\hb)\newtext{,}
\end{equation}
\newtext{where} $\Psi$ is a function from $\eR^N$ to $\eR$, playing the role of a regularization function.
 \correcEC{This penalty} function may \correcEC{for instance} be useful to incorporate some prior knowledge about the sought parameter vector $\hb$, e.g. some sparsity requirement, possibly in some transformed domain. \correcEC{In this paper, a family of differentiable, non necessarily convex, regularization functions  \cite{Chouzenoux2012} is considered.} 
Problem \eqref{e:probdet} is encountered in numerous applications such as system identification, channel equalization, linear prediction or interpolation, 
echo cancellation, interference removal, and supervised classification. In the latter area, $(\Xv_n)_{n\ge 1}$ are vectors ($Q=1$) which may correspond
to features obtained through some nonlinear mapping of the data to be classified in a given training sequence, and $(\yv_n)_{n\newtext{\ge 1}}$ may be 
the associated (discrete-valued) class index vector~\cite{Bottou2004,ML2011,Theodoridis_S_2015_Machine_L}.
Although some other measures (e.g. the logistic regression function) are often more effective in this context, the use of a least squares
criterion may still be competitive for simplicity reasons~\cite{Rifkin03,Rosasco_L_2016_GURLS}, while the regularization term serves here to avoid overfitting which could arise
when the number of extracted features is large~\cite{Rosasco2007}. Signal reconstruction constitutes another application field of interest. Then, the vector $\hb$ corresponds to an unknown signal \newtext{related to some measurements} $(\yv_n)_{n\ge 1}$ obtained through
products with \newtext{matrices} $(\Xv_n^\top)_{n\ge 1}$, and additionally \newtext{corrupted by some noise process}  \cite{Peyrera15,chapitre_livre_wiley, DinizBook}. Each \newtext{matrix} $\Xv_n^\top$ with \correcEC{$n\in \eN \setminus \left\{0\right\}$} corresponds to \newtext{$Q$ lines} of the \newtext{full} acquisition matrix and it is here considered as random. Under suitable stationarity assumptions, the classical
least squares data fidelity term can be modeled as $\EE\big(\|\yv_n-\Xv_n^\top \hb\|^2\big)/2$, whereas due to the ill-posedness
of the great majority of such inverse problems, a regularization term $\Psi$ needs to be introduced so as to obtain reliable estimates \cite{demoment1989image}.

Many optimization algorithms can be devised to solve Problem \eqref{e:probdet} depending on the assumptions made on $\Psi$ 
\cite{Nocedal99,Combettes_2010,Afonso2011,Boyd_S_2011_j-found-tml_distributed_osl_admm}. \newtext{
In this work, we will be interested in Majorize-Minimize (MM) algorithms \cite{Hunter04,Zhang2007}. In such approaches, the iterates result from successive minimizations of simple surrogates (e.g. quadratic surrogates) majorizing the cost-function. MM algorithms are very flexible and benefit from good theoretical
and practical convergence properties. However, the computation load resulting from the minimization of the majorant function may be prohibitive in the context of large scale problems. The strategy we will adopt in this work is to account for subspace acceleration \cite{Zibulevsky10}, i.e., to constrain the inner minimization step to a subspace of low dimension, typically restricted to the gradient computed at the current iterate and to a memory part (e.g. the difference between the current iterate and a previous one). In a number of recent works \cite{Chouzenoux2011a,Chouzenoux2012,Florescu2014}, MM subspace algorithms provide fast numerical solutions to optimization problems involving smooth functions}, in particular in the case of large-scale problems. Note that, although our approach will \newtext{require} that $\Psi$ is a differentiable function, it has been shown that tight approximations of nonsmooth penalizations such as 
$\ell_1$ (resp. $\ell_0$) functions, namely $\ell_2-\ell_1$ (resp. $\ell_2-\ell_0$) functions, can be employed and are often quite effective in practice \cite{Chouzenoux2012,Florescu2014}. Another advantage of the 
class of optimization methods under investigation is that their convergence can be established under some technical assumptions,
even in the case when $\Psi$ is a nonconvex function (see \cite{Chouzenoux2012} for more details).

One of the difficulties encountered in machine learning or adaptive processing is that  Problem~\eqref{e:probdet} cannot be directly solved since
the second-order statistical moments $\varrho$, $\rb$ and $\Rb$ are often unknown a priori or their direct computation is too intensive, and they have thus to be estimated online.
In the simple case when $\Psi=0$, the classical Recursive Least Squares (RLS) algorithm can be used for this purpose \cite{Haykin2002}. 
When $\Psi$ is nonzero, stochastic approximation algorithms have been developed such as the celebrated stochastic gradient descent (SGD) algorithm \cite{Robbins1951,Ermo67,Guse71,Bertsekas_D_20000-siopt_Gradient_cgme} and some of its proximal extensions \cite{Atchade_Y_2014_Stochastic_pgga,Rosasco_L_2014_Convergence_spga,Combettes_P_2016_pafa_2016_Stochastic_apfbsmo,Konecny2014}. 
The convergence speed of SGD may be relatively slow so that
various extensions of it have been developed to alleviate this problem (see \cite{Polyak1992,Nemirovski08,Bach2011,Peyrera15} and the references therein).
Many efforts have also been devoted to developing adaptive variants of this algorithm \cite{Widrow1985,Macchi_0_1995_book_adaptive_plmsa}, in particular when identifying filters having sparse impulse responses
(see e.g. \cite{Hoshuyama_O_2004_p-icassp_Generalized_pvssa,Khong_A_2006_p-asilomar_Efficient_usaf,Chen_Y_2009_p-icassp_Sparse_lms,Paleologu2010,Murakami2010,Meng_R_2011_procssdp_sparsity_aapaasi,Markus_L_2013_p-icassp_Affine_passi}). In addition, in \cite{Kopsinis_Y_2011_ieee-tsp_online_ssisr}, a set theoretic approach is adopted for online sparse estimation based on projections onto weighted $\ell_1$ balls, which is extended in \cite{Slavakis_K_2013_ieee-tsp_Generalized_tosal} by making use of generalized thresholding mappings.
It is worth noting that a sparse RLS algorithm was proposed in \cite{Babadi_B_2010-ieee-tsp_SPARSELS_srlsa} for complex-valued signals in the case when $\Psi$ is an $\ell_1$ norm.
An online variant of the RLS algorithm corresponding to a time weighted LASSO estimator was also designed in \cite{Angelosante2010} which relies on a coordinate descent approach.
A similar problem was also addressed in \cite{Themelis_K_2014_ieee-tsp_Variational_bfsae} by adopting a novel Bayes variational approach, for which weak theoretical convergence guarantees however exist.
If we except \cite{Ono_S_2013_p-icassp_sparse_siu} where an adaptive primal-dual splitting is employed to deal with a total variation penalization, in almost all the works on sparse adaptive filtering,
the sparsity is directly imposed on the filter coefficients, without introducing any linear transform of them.

Designing Majorize-Minimize optimization algorithms in a stochastic context constitutes a challenging task since most of the existing works concerning these methods 
have been focused on batch optimization procedures, and the related convergence proofs usually rely on deterministic
tools. We can however mention a few recent works \cite{Xiao2010,Mairal2013,Carlson16} where stochastic MM algorithms
have been investigated for general loss functions under specific assumptions (e.g. the independence of the involved random variables \cite{Xiao2010,Mairal2013}),
but without introducing any search subspace. Works which are more closely related to ours are those based on Newton or quasi-Newton
stochastic algorithms \cite{Birge1995,Bordes2009,Yu2010,Byrd2014,Gower2016a,Gower2016b}, in particular the approaches in \cite{Yu2010,Byrd2014} provide extensions of BFGS algorithm, but proving the convergence of these algorithms requires some specific assumptions. Like BFGS \newtext{approaches}, MM subspace methods
use a memory of previous estimates so as to accelerate the convergence.

\correcEC{Our main contributions in this paper are:
\begin{itemize}
	\item[$\bullet$] to propose an online version of the MM subspace algorithm from~\cite{Chouzenoux2011a,Chouzenoux2012}, for a wide class of penalized least squares problems,
	\item[$\bullet$] to derive a recursive form, with reduced complexity, of the resulting online MM subspace method, 
	\item[$\bullet$] to prove the convergence of the iterates produced by our method in the stochastic context,
	\item[$\bullet$] to show the good practical performance of this method when it is combined with a memory gradient subspace.
\end{itemize}
}

In Section \ref{se:probform}, we show how Problem \eqref{e:probdet} can be reformulated in a learning context. 
The MM strategy which is proposed in this work is described in Section \ref{s:Mprop}.
In Section~\ref{se:recur}, we give the form of the resulting recursive algorithm and, in Section \ref{se:complex}, we evaluate
its computational complexity. A convergence analysis of the proposed stochastic Majorize-Minimize subspace algorithm is performed in Section \ref{sec:convergence}. In Section \ref{se:ident}, two simulation examples in the context of filter identification illustrate the good performance of our algorithm when a memory gradient subspace is employed. Some conclusions are drawn in Section \ref{se:conclu}.

\section{Problem formulation}\label{se:probform}

\begin{table}[t]
\tiny
\centering
\caption{Smooth penalty functions $\psi_s$ \correcEC{fulfilling Assumption~\ref{a:basic}} and their associated weighting functions $\nu_s$. 
All expressions are valid for $t \in \eR$, $(\lambda_s,\delta_s)\in (0,+\infty)^2$ and $\kappa_s\in[1,2]$.}
\label{t:penalty}
\begin{tabular}{|p{0.3cm}|c|c|c|c|}
\cline{2-5}
\multicolumn{1}{c|}{ } & $\lambda_s^{-1} \psi_s(t)$ & $\lambda_s^{-1} \nu_s(t)$ & Type & Name\\
\hline
\hline
&$|t|-\delta_s \log(|t|/\delta_s +1)$ & $(|t|+\delta_s)^{-1}$ & $\ell_2-\ell_1$ & \\
\cline{2-5}
&$\begin{cases} t^2 & \mbox{if} \, |t| < \delta_s \\ 2 \delta_s |t| - \delta_s^2 & \mbox{otherwise} \end{cases}$ & 
$\begin{cases}  2 & \mbox{if} \, |t| < \delta_s \\ 2 \delta_s/ |t| & \mbox{otherwise} \end{cases}$ & $\ell_2-\ell_1$ & Huber\\
\cline{2-5}
\rotatebox{90}{\parbox{3mm}{\multirow{3}{*}{Convex}}} 
&$\log(\operatorname{cosh}(t))$ & $\begin{cases} \operatorname{tanh}(t) / t & \mbox{if} \, t \neq 0 \\ 1 & \mbox{otherwise} \end{cases}$ & $\ell_2-\ell_1$ & Green\\
\cline{2-5}
&$(1 + t^2 / \delta_s^2)^{\kappa_s/2} - 1$ & $ \kappa_s \delta_s^{-2} (1 + t^2 / \delta_s^2)^{\kappa_s/2 - 1}$ & $\ell_2-\ell_{\kappa_s}$ & \\ 
\hline
\hline
&$1 - \exp(-t^2 / (2 \delta_s^2))$ & $\delta_s^{-2} \exp(-t^2 / (2 \delta_s^2)) $ & $\ell_2-\ell_0$ & Welsch \\
\cline{2-5}
&$t^2 / (2 \delta_s^2 + t^2)$ & $4\delta_s^2 / (2 \delta_s^2 + t^2)$ & $\ell_2-\ell_0$ & Geman\\
& & & & -McClure \\
\cline{2-5}
&$\begin{cases} 1 - (1-t^2 / (6 \delta_s^2))^3 & \mbox{if} \, |t| \leq \sqrt{6} \delta_s \\
1 & \mbox{otherwise}
\end{cases}
$ 
&
$
\begin{cases} \delta_s^{-2} (1 - t^2 / (6 \delta_s^2))^2 & \mbox{if} \, |t| \leq \sqrt{6} \delta_s \\
0 & \mbox{otherwise}
\end{cases}
$ & $\ell_2-\ell_0$ & Tukey biweight \\
\cline{2-5}
\rotatebox{90}{\parbox{3mm}{\multirow{3}{*}{Nonconvex}}} 
&$ \operatorname{tanh}(t^2 / (2 \delta_s^2))$ & $\delta_s^{-2} (\operatorname{cosh}(t^2 / (2 \delta_s^2))^{-2}$ & $\ell_2-\ell_0$ &  Hyberbolic\\
& & & & tangent \\
\cline{2-5}
&$\log(1 + t^2 / \delta_s^2)$ & $2 / (t^2 + \delta_s^2)$ & $\ell_2-\log$ & Cauchy \\
\cline{2-5}
&$1 - \exp(1-(1 + t^2 / (2\delta_s^2))^{\kappa_s/2})$ & $(\kappa_s/(2 \delta_s^2)) (1 + t^2 / (2\delta_s^2))^{\kappa_s/2 - 1} \exp(1-(1 + t^2 / (2\delta_s^2))^{\kappa_s/2})$ & $\ell_2 - \ell_{\kappa_s} - \ell_0$ & Chouzenoux\\
\hline
\end{tabular}
\end{table}

In a learning context, function $F$ can be replaced by a sequence $(\mathrm{F}_n)_{n\ge 1}$ of stochastic approximations of it, which are defined as
follows: for every \correcEC{$n\in \eN \setminus \left\{0\right\}$},
\begin{align}\label{e:defFn}
(\forall \hb\in \eR^N)\quad
 \mathrm{F}_n(\hb) &= \frac{1}{2\overline{\vartheta}_n}\sum_{k=1}^n \vartheta^{n-k}\|\yv_k-\Xv_k^\top \hb\|^2 + \Psi(\hb)\nonumber\\
 &= \frac12 \rho_n - \rv_n^\top \hb + \frac12 \hb^\top \Rv_n \hb  + \Psi(\hb),
\end{align}
where $\vartheta \in (0,1)$,
\begin{equation}\label{e:defbarvartheta}
\overline{\vartheta}_n = \sum_{k=0}^{n-1} \vartheta^{k} =
\begin{cases}
n & \mbox{if $\vartheta = 1$}\\
\displaystyle\frac{1-\vartheta^n}{1-\vartheta} & \mbox{if $\vartheta \in (0,1)$,}
\end{cases}
\end{equation}
and $\rho_n$, $\rv_n$, and $\Rv_n$ are given by
\begin{align}
\rho_n & = \frac{1}{\overline{\vartheta}_n} \sum_{k=1}^n \vartheta^{n-k} \|\yv_k\|^2 \label{e:defsamprho}\\
\rv_n  & = \frac{1}{\overline{\vartheta}_n} \sum_{k=1}^n \vartheta^{n-k} \Xv_k \yv_k \label{e:defsampr}\\
\Rv_n  & = \frac{1}{\overline{\vartheta}_n} \sum_{k=1}^n \vartheta^{n-k} \Xv_k \Xv_k^\top \label{e:defsampR}.
\end{align}
In the case when $\vartheta = 1$, we retrieve
the classical sample estimates of $\varrho$, $\rb$, and $\Rb$. When $\vartheta \in (0,1)$, it can be interpreted
as an exponential forgetting factor \cite{Haykin2002} which may be useful in adaptive processing scenarios (see Section \ref{sec:adapt}).

Hereafter, we will assume that the regularization function $\Psi$ has the following form:
\begin{equation}\label{e:defPsi}
 (\forall \hb \in \eR^N)\quad \Psi(\hb) = \frac12 \hb^\top \Vb_0 \hb - \vb_0^\top \hb +\sum_{s=1}^S \psi_s(\|\Vb_s \hb - \vb_s\|)
\end{equation}
where $\vb_0 \in \eR^N$, $\Vb_0\in \eR^{N\times N}$ is a symmetric positive semi-definite matrix, and, for every $s\in \{1,\ldots,S\}$,
$\vb_s\in \eR^{P_s}$, $\Vb_s \in \eR^{P_s\times N}$, and $\psi_s \colon \eR \to \eR$ is a smooth function.
The first \correcEC{ two terms in \eqref{e:defPsi}} can be viewed as an elastic net penalty \cite{Zou05}, while
various choices can be made for the \correcEC{last term}. As shown in Table \ref{t:penalty}, in addition to quadratic
regularization functions (obtained when $S=1$ and $\psi_1 = 0$), $\ell_2-\ell_1$ functions and smoothed $\ell_2-\ell_0$ functions constitute standard choices. The matrices $(\Vb_s)_{1\le s\le S}$ may be set to identity or they may serve to model possible transforms or discrete differentiation operators, and vectors $(\vb_s)_{1\le s \le S}$ may be used to define reference values.

Note that the regularization strategy adopted in \cite{Angelosante2010} amounts to replacing $\Psi$ in \eqref{e:defFn} by $\lambda_n \overline{\Psi}$ where
$\overline{\Psi}$ is a (possibly weighted) $\ell_1$ norm and $\lambda_n\in [0,+\infty)$. Consistency results can then be established under the
assumption that $\vartheta = 1$ and $\lim_{n\to+\infty} \lambda_n = 0$. Our approach here is different, not only because we are interested in a wide class
of regularization functions, but also in the sense that we are looking for a solution
to the fully regularized problem \eqref{e:probdet} instead of a solution to the mean square criterion.

Our objective in the next section will be to propose an efficient recursive method for
minimizing functions $(\mathrm{F}_n)_{n\ge 1}$.

\section{Proposed method} \label{se:propmeth}
\subsection{Majorization property}\label{s:Mprop}
At each iteration \correcEC{$n\in \eN \setminus \left\{0\right\}$}, we propose to replace $\mathrm{F}_n$ by a surrogate function $\Theta_n(\cdot,\hv_n)$
based on the current estimate $\hv_n$ (computed at the previous iteration). More precisely, a tangent majorant function is chosen
such that
\begin{align}
& (\forall \hb \in \eR^N)\qquad \mathrm{F}_n(\hb) \le \Theta_n(\hb,\hv_n) \label{e:MMmaj}\\
& \mathrm{F}_n(\hv_n) = \Theta_n(\hv_n,\hv_n).
\end{align}
For the so-defined MM strategy to be worthwhile, the surrogate function has to be built in such a way that its minimization is simple.
For this purpose, the following assumptions will be made on the regularization function $\Psi$ defined in \eqref{e:defPsi}:
\begin{assumption}\label{a:basic}\ 
\begin{enumerate} 
\item \label{a:basicii} For every $s \in \left\{1,\ldots,S\right\}$, $\psi_{s}$ is an even lower-bounded function, which is continuously differentiable, and
$\lim_{\underset{t\neq 0}{t \to 0}}  \dot{\psi}_{s}(t)/t \in \eR$, where $\dot{\psi}_{s}$ denotes the
derivative of $\psi_{s}$.
\item \label{a:basiciii} For every $s\in \left\{1,\ldots,S\right\}$, $\psi_{s}(\sqrt{.})$ is concave on $[0,+\infty)$.
\item \label{a:basiciv}  There exists $\overline{\nu} \in [0,+\infty)$ such that
$(\forall s \in \{1,\ldots,S\})$
$(\forall t \in [0,+\infty))$ $0 \leq  \nu_s(t) \leq
\overline{\nu}$, where
$\nu_s(t) = \dot{\psi}_{s}(t)/t$.\footnote{The function is extended by continuity when $t=0$.}
\end{enumerate}
\end{assumption}
These assumptions are satisfied by a wide class of functions $\Psi$\correcEC{~\cite{Nikolova05}}, in particular 
those corresponding to the choices of the potential functions $(\psi_s)_{1\le s \le S}$ listed in Table \ref{t:penalty}.

\correcEC{
Assumption~\ref{a:basic} implies that each function $\psi_{s}$ is majorized at every $t \in \eR$, by a quadratic function, such that
\begin{equation}\label{e:ineqMMpsis}
(\forall t' \in \eR) \; \psi_{s}(t') \leq \psi_{s}(t) + \dot{\psi}_{s}(t)(t'-t) + \frac{1}{2} \nu_s(|t|)(t'-t)^2.
\end{equation}
Note that the above inequality is at the core of iterative reweighted least-squares algorithms~\cite{Bissantz09} and of half quadratic methods~\cite{Allain_M_2006_ieee-tip_On_galc} for the minimization of penalized quadratic functions. The following majorization then straightforwardly results from \eqref{e:ineqMMpsis}:
\begin{proposition}\label{p:majquad}
Under Assumption \ref{a:basic}, for every \correcEC{$n\in \eN \setminus \left\{0\right\}$} and $\hb\in \eR^N$,
a tangent majorant of $\mathrm{F}_n$ at $\hb$ is
\begin{align}
\label{e:defQc}
(\forall \hb' \in \eR^N)\quad
\Theta_n(\hb',\hb) =\;
&\mathrm{F}_n(\hb)+\nabla \mathrm{F}_n(\hb)^\top (\hb'-\hb) \nonumber\\
&+ \frac12 (\hb'-\hb)^\top \Av_n(\hb) (\hb'-\hb),
\end{align}
where $\Av_n(\hb)$ is given by 
\begin{align}
&\Av_n(\hb) = \Rv_n+\Vb_0+ \Vb^\top\Diag\big(\bb(\hb)\big)\Vb \in \eR^{N\times N}\label{e:defA}\\	
&\Vb = [\Vb_1^\top\ldots\Vb_S^\top]^\top\in \eR^{P\times N} \label{e:defVb}\\
&\vb = [\vb_1^\top\ldots\vb_S^\top]^\top\in \eR^{P}
\end{align}
with $P = P_1+\cdots+P_S$, and
$\bb(\hb)= \big(b_i(\hb)\big)_{1\le i \le P}\in \eR^{P}$ is such that
\begin{equation}\label{e:defbb}
\bb(\hb) = \left[\nu_{1}(\| \Vb_1 \hb-\vb_1\|) \unb_{P_1}^\top \,\ldots\, \nu_{S}(\| \Vb_S \hb-\vb_S\|) \unb_{P_S}^\top\right]^\top,
\end{equation}
where $\unb_P \in \eR^P$ denotes a vector of size $P$ with all entries equal to one.
\end{proposition}
If, we define, for every $n\in \eN \setminus \left\{0\right\}$, $\hv_{n+1}$ as the minimizer of $\Theta_n(\cdot,\hv_n)$, we obtain an online 
form of a half-quadratic algorithm \cite{Allain_M_2006_ieee-tip_On_galc}. Half-quadratic algorithms are known to be effective batch optimization methods, but the use of such method requires the inversion of matrix $\Av_n(\hv_n)$ at each iteration $n$,
which may be intractable in the context of large scale problems. Subsequently, we propose a subspace acceleration
strategy so as to reduce the computational cost of the proposed method.
}

\correcEC{
\subsection{Subspace acceleration strategy}
The main idea of subspace acceleration is to restrict the minimization space to a subspace spanned by a small
number of vectors, instead of minimizing the majorant over the whole space. The proposed MM subspace algorithm consists of defining the following sequence of random vectors $(\hv_n)_{n\ge 1}$:
\begin{equation}\label{e:MMmini}
(\forall \correcEC{n\in \eN \setminus \left\{0\right\}}) \qquad \hv_{n+1} \in \argmin_{\hb \in \ran\Dv_n} \Theta_n(\hb,\hv_n),
\end{equation}
where $\hv_1$ has to be set to an initial value, and \correcEC{$\ran\Dv_n$ denotes the range of a matrix $\Dv_n\in \eR^{N\times M_n}$ that} should satisfy the above assumption: 
\begin{assumption}
For every \correcEC{$n\in \eN \setminus \left\{0\right\}$, $\{\nabla \mathrm{F}_n(\hv_n),\hv_n\} \subset \ran \Dv_n$}.
\label{ass:subspace}
\end{assumption}
Several approaches can be considered to construct $\Dv_n$ fulfilling Assumption~\ref{ass:subspace}~\cite[Tab.I]{Chouzenoux2011a}. The simplest choice is to set $\Dv_n = [-\nabla \mathrm{F}_n(\hv_n),\hv_n]$, so that \eqref{e:MMmini} reads
\begin{equation}
\hv_{n+1} = \mathrm{u}_{n,2} \hv_n-\mathrm{u}_{n,1} \nabla \mathrm{F}_n(\hv_n),
\end{equation}
where $(\mathrm{u}_{n,1},\mathrm{u}_{n,2})$ is a pair of real-valued random variables. In the special case when $\mathrm{u}_{n,2}=1$, we recover the form of a SGD-like algorithm with step-size $\mathrm{u}_{n,1}$. In the machine learning literature, various forms of the step-size for SGD have been proposed \cite{Bach2011}, which often require to tune up some parameters (e.g. a multiplicative factor) so as to get the best convergence profile on the available dataset. On the contrary, the MM strategy allows us to automatically adjust $(\mathrm{u}_{n,1},\mathrm{u}_{n,2})$ at each iteration. Another possibility is to take, for every $n\in \eN \setminus \left\{0\right\}$, $\ran \Dv_n =\eR^N$. In that case, we recover the online half-quadratic method mentioned earlier, which may have a high computational cost. A more efficient strategy that is at the roots of many works in the context of batch optimization is to adopt an intermediate size subspace matrix, gathering the gradient subspace $[-\nabla \mathrm{F}_n(\hv_n),\hv_n]$ complemented with few vectors containing information regarding the previous iterates (e.g., previous gradient directions, previous iterates,...)~\cite{Bioucas2007,Beck09,Nesterov2007}. In particular, the memory gradient subspace~\cite{Miele69}, defined as:
\begin{equation}\label{e:2memorysub}
\Dv_n = 
\begin{cases}
[-\nabla \mathrm{F}_n(\hv_n),\hv_n,\hv_n-\hv_{n-1}] & \mbox{if $n > 1$}\\
[-\nabla \mathrm{F}_n(\hv_1),\hv_1] & \mbox{if $n = 1$,}
\end{cases}
\end{equation}
was observed to lead to fast convergence on several examples in the field of signal and image restoration~\cite{Chouzenoux2013icip,Florescu2014}. 
} 

\subsection{Recursive MM strategy}\label{se:recur}

\correcEC{
We derive in this section a recursive form of the proposed stochastic MM subspace algorithm in \eqref{e:MMmini}, with the objective
to limit its complexity. First, note that, according to \eqref{e:defFn}, \eqref{e:defPsi}, and the definition of functions $(\nu_s)_{1\le s\le S}$ in Assumption~\ref{a:basic}\ref{a:basiciv}, for every \correcEC{$n\in \eN \setminus \left\{0\right\}$}, the gradient of $\mathrm{F}_n$ is given by
\begin{equation}\label{e:gradFn}
(\forall \hb \in \eR^N)\qquad \nabla \mathrm{F}_n(\hb) = \Av_n(\hb)\hb - \cv_n(\hb),
\end{equation}
where 
\begin{equation}
\cv_n(\hb) = \rv_n + \vb_0 + \Vb^\top\Diag\big(\bb(\hb)\big) \vb \in \eR^{N} \label{e:defcv}.
\end{equation}
Thus, using~\eqref{e:defQc}, we can rewrite~\eqref{e:MMmini} as
\begin{equation}\label{e:hnp1Dnun}
\hv_{n+1} = \Dv_n \uv_n,
\end{equation}
 where $\uv_n$ is an $\eR^{M_n}$-valued random vector such that:
\begin{align}\label{e:defun}
\uv_{n} &= \Bv_n^{\dagger} \Dv_n^\top \big(\Av_n(\hv_n) \hv_n - \nabla \mathrm{F}_n(\hv_n)\big)\nonumber\\
&= \Bv_n^{\dagger} \Dv_n^\top \cv_n(\hv_n),
\end{align}
with
\begin{equation} \label{e:defBn}
\Bv_n = \Dv_n^\top \Av_n(\hv_n) \Dv_n
\end{equation}
and $(\cdot)^\dagger$ denoting the pseudo-inverse operation. It is important to note that, as $\Bv_n$ is of dimension $M_n\times M_n$ where $M_n$ is small
(typically $M_n=3$ for the choice of the subspace in \eqref{e:2memorysub} when $n>1$), this pseudo-inversion is light. This constitutes the key 
advantage of the proposed approach. 
}

\correcEC{
By using \eqref{e:defbarvartheta}, \eqref{e:defsampr} and \eqref{e:defsampR}, the following recursive updates of $(\rv_n)_{n\ge 1}$ and $(\Rv_n)_{n\ge 1}$,
can be performed
\begin{align}
(\forall n \in \eN\setminus \{0\}) \qquad  \rv_n &= \rv_{n-1}+ \frac{1}{\overline{\vartheta}_n} (\Xv_n \yv_n-\rv_{n-1})\\
\Rv_n &= \Rv_{n-1}+ \frac{1}{\overline{\vartheta}_n} (\Xv_n \Xv_n^\top-\Rv_{n-1}) \label{e:recurRvn},
\end{align}
where we have set $\rv_0 = \zerob$ and $\Rv_0 = \Ob_N$ and we have used the identity:
$\vartheta \overline{\vartheta}_{n-1}/\overline{\vartheta}_{n} = 1- \overline{\vartheta}_{n}^{-1}$. Then, it follows from \eqref{e:defA}, \eqref{e:defBn} and \eqref{e:recurRvn} that
\begin{equation}
\correcEC{(\forall n\in \eN \setminus \left\{0\right\})}  \; \Bv_n = \Dv_n^\top \big(\Dv^\Rv_n+\Dv^{\Vb_0}_n\big)
+\big(\Dv^{\Vb}_n\big)^\top\Diag\big(\bb(\hv_n)\big)\Dv^{\Vb}_n,
\end{equation}
where
\begin{align}
\correcEC{(\forall n\in \eN \setminus \left\{0\right\})} \qquad &\Dv^\Rv_n = \Rv_n \Dv_n \in \eR^{N\times M_n} \label{e:defDvRv}\\
& \Dv^{\Vb_0}_n = \Vb_0 \Dv_n \in \eR^{N\times M_n}\\
& \Dv^\Vb_n = \Vb \Dv_n \in \eR^{P\times M_n} \label{e:defDvVb}.
\end{align}
Finally, let us assume, without loss of generality, that the algorithm is initialized with $\hv_1 = \Dv_0 \uv_0$, where $\Dv_0\in \eR^{N\times M_0}$ and $\uv_0\in \eR^{M_0}$. Then, \eqref{e:gradFn} and \eqref{e:hnp1Dnun} yield
\begin{equation}
\correcEC{(\forall n\in \eN \setminus \left\{0\right\})}  \qquad \nabla\mathrm{F}_n(\hv_n) = \Dv^\Av_{n-1}\uv_{n-1}-\cv_n(\hv_n), 
\end{equation}
where we have set
\begin{equation}
(\forall n \in \eN) \qquad \Dv^\Av_n = \Av_{n+1}(\hv_{n+1})\,\Dv_n \in \eR^{N\times M_n}.
\end{equation}
By using \eqref{e:defA}, \eqref{e:recurRvn} and \eqref{e:defDvRv}-\eqref{e:defDvVb}, the latter variable can be reexpressed as
\begin{align}
\Dv^\Av_n = &\;\Rv_{n+1} \Dv_n + \Dv^{\Vb_0}_n + \Vb^\top \Diag\big(\bb(\hv_{n+1})\big) \Dv^{\Vb}_n\nonumber\\
= &\;(1-\frac{1}{\overline{\vartheta}_{n+1}}) \Dv^\Rv_n + \frac{1}{\overline{\vartheta}_{n+1}} \Xv_{n+1} (\Xv_{n+1}^\top \Dv_n)+ \Dv^{\Vb_0}_n\nonumber\\ 
& + \Vb^\top \Diag\big(\bb(\hv_{n+1})\big) \Dv^{\Vb}_n.
\end{align}
The resulting relations are summarized in Algorithm~\ref{e:stochMMalg}, which can be understood as a recursive implementation of Algorithm~\eqref{e:MMmini}.
}

\begin{algorithm}
\caption{Stochastic MM subspace method}\label{e:stochMMalg} 
\DontPrintSemicolon
$\rv_0 = \zerob, \Rv_0 =  \Ob_N$\;
Initialize $\Dv_0, \uv_0$\;
$\hv_1 = \Dv_0\uv_0, \Dv^\Rv_{0} = \Ob_{N\times M_n}, \Dv^{\Vb_0}_{0} = \Vb_0\Dv_0, \Dv^{\Vb}_{0} = \Vb\Dv_0$\;
\For{$n = 1,\ldots$}{
\nl \label{step1} $\rv_n = \rv_{n-1}+ \frac{1}{\overline{\vartheta}_n} (\Xv_n \yv_n-\rv_{n-1})$\; 
\nl \label{step2} $\cv_n(\hv_n) = \rv_n + \vb_0 + \Vb^\top\Diag\big(\bb(\hv_n)\big) \vb$\; 
\nl \label{step3} $\Dv^\Av_{n-1} = (1-\frac{1}{\overline{\vartheta}_n}) \Dv^\Rv_{n-1} + \frac{1}{\overline{\vartheta}_n} \Xv_n (\Xv_n^\top \Dv_{n-1})$\;
$\qquad\quad\;\;\; + \Dv^{\Vb_0}_{n-1} + \Vb^\top \Diag\big(\bb(\hv_n)\big) \Dv^{\Vb}_{n-1}$\; 
\nl \label{step4} $\nabla\mathrm{F}_n(\hv_n) = \Dv^\Av_{n-1}\uv_{n-1}-\cv_n(\hv_n)$\; 
\nl \label{step5} $\Rv_n = \Rv_{n-1}+ \frac{1}{\overline{\vartheta}_n} (\Xv_n \Xv_n^\top-\Rv_{n-1})$\; 
\nl Set $\Dv_n$ using $\nabla\mathrm{F}_n(\hv_n)$\;
\nl \label{step7} $\Dv^\Rv_n = \Rv_n \Dv_n, \Dv^{\Vb_0}_n = \Vb_0\Dv_n, \Dv^{\Vb}_n = \Vb\Dv_n$\; 
\nl \label{step8} $\Bv_n = \Dv_n^\top \big(\Dv^\Rv_n+\Dv^{\Vb_0}_n\big)+\big(\Dv^{\Vb}_n\big)^\top\Diag\big(\bb(\hv_n)\big)\Dv^{\Vb}_n$\;
\nl \label{step9} $\uv_{n} = \Bv_n^{\dagger} \Dv_n^\top \big(\cv_n(\hv_n)\big)$\; 
\nl \label{step10} $\hv_{n+1} = \Dv_n \uv_{n}$\; 
}
\end{algorithm}

\subsection{Complexity}\label{se:complex}
\correcEC{Provided that the subspace dimensions $(M_n)_{n\in \eN}$ are small, Algorithm~\ref{e:stochMMalg} has a low complexity, as shown in Table \ref{t:complexity}.}

\begin{table}
\centering
\caption{Complexity in terms of multiplications for iteration $n$ of Algorithm~\ref{e:stochMMalg}.}
\label{t:complexity}
\begin{tabular}{|c|c|c|}
\hline
Step & Complexity &  Complexity\\
& for $\Vb\in \eR^{P\times N}$ arbitrary & when $\Vb = \Ib_N$\\
\hline
\hline
\ref{step1} & \multicolumn{2}{c|}{$N(Q+1)$}\\
\hline
\ref{step2} & $(N+1)P$ & $N$\\ 
\hline
\ref{step3} & $M_{n-1}\big(N(2Q+P+1)+P+Q\big)$ & $M_{n-1}\big(N(2Q+1)+Q\big)$\\ 
\hline
\ref{step4} & \multicolumn{2}{c|}{$N M_{n-1}$}\\
\hline
\ref{step5} & \multicolumn{2}{c|}{$N(N+1)Q/2$}\\ 
\hline
\ref{step7} & $NM_n(2N+P)$ & $2N^2M_n$\\ 
\hline
\ref{step8} & $M_n\big((M_n+1)(N+P)/2+P\big)$ & $N M_n (M_n+3)/2$\\
\hline
\ref{step9} & \multicolumn{2}{c|}{$O(M_n^3)+M_n(N+M_n)$}\\
\hline
\ref{step10} & \multicolumn{2}{c|}{$N M_n$}\\
\hline
\end{tabular}
\end{table}
Indeed, the global complexity of a direct implementation of Algorithm~\ref{e:stochMMalg}, evaluated in terms of multiplications
at iteration $n$, is of the order of  
\[
N\big(P(M_n+M_{n-1}+1)+N(4M_n+Q)/2\big),
\]
if we assume that $N\gg \max\{M_n,M_{n-1},Q\}$.
The first term $N P(M_n+M_{n-1}+1)$ corresponds to an upper bound on the complexity induced by the use of matrices $(\Vb_s)_{1\le s \le S}$
within the regularization term.  \correcEC{Note that these} matrices often have a sparse structure (in particular when discrete derivative operators are employed) which \correcEC{may lead} to a much lower computational cost. \correcEC{Moreover, when} $\Vb=\Ib_N$, the identity matrix of $\eR^N$, which is a scenario frequently encountered in adaptive filtering, this term merely vanishes in the evaluation of the global complexity.

The computational complexity can also be reduced by taking advantage of the specific form of matrices $(\Dv_n)_{n\ge 1}$. 
\correcEC{Here, we focus our analysis on the example of the memory gradient subspace defined in~\eqref{e:2memorysub} \correcEC{although it should be noticed that the  ideas hereinbelow could be easily generalized to a wide class of subspaces where matrices $(\Dv_n)_{n\ge 1}$ represent memory features~(e.g. \cite[Tab. II]{Chouzenoux2011a}.)}. For the particular case of subspace~\eqref{e:2memorysub}, we have:}
\begin{equation}\label{e:recursDVn}
(\forall n > 1)\quad \Dv^\Vb_n = [-\Vb\nabla\mathrm{F}_n(\hv_n),\Vb\hv_n,\Vb\hv_n-\Vb\hv_{n-1}].
\end{equation}
\correcEC{Since, for every $n\ge 1$,}
\begin{equation}
\Vb\hv_n = \Vb \Dv_{n-1}\uv_{n-1} = \Dv^\Vb_{n-1} \uv_{n-1},
\end{equation}
a recursive formula holds to compute the last two components of $\Dv^\Vb_n$ in \eqref{e:recursDVn}. 
The initial complexity of $3NP$ multiplications is thus reduced to $N(P+3)$.
Similar recursive procedures can be employed to compute $(\Dv^{\Vb_0}_n)_{n>1}$
allowing the complexity to be reduced to $N(N+3)$ from $3N^2$.
In addition, we have, for every $n > 1$,
\begin{equation}\label{e:recursDRn}
\Dv^\Rv_n = [-\Rv_n\nabla\mathrm{F}_n(\hv_n),\hv^\Rv_n,\hv^\Rv_n-\Rv_n\hv_{n-1}],
\end{equation}
where, by using \eqref{e:recurRvn},
\begin{align}
\hv^\Rv_n =
\Rv_n\hv_n &= (1-\frac{1}{\overline{\vartheta}_n}) \Rv_{n-1}\hv_n + \frac{1}{\overline{\vartheta}_n} \Xv_n\Xv_n^\top \hv_n\nonumber\\
& = (1-\frac{1}{\overline{\vartheta}_n}) \Dv^\Rv_{n-1}\uv_{n-1} + \frac{1}{\overline{\vartheta}_n} \Xv_n\Xv_n^\top \hv_n\\
\Rv_{n}\hv_{n-1} &= (1-\frac{1}{\overline{\vartheta}_n})\hv^\Rv_{n-1} + \frac{1}{\overline{\vartheta}_n} \Xv_n\Xv_n^\top \hv_{n-1}.
\end{align}
It can be further observed that last term $(\overline{\vartheta}_n)^{-1}\Xv_n\Xv_n^\top \hv_{n-1}$
has already been computed in Step \ref{step3} of Algorithm~\ref{e:stochMMalg}. Therefore,
instead of $3N^2$ multiplications, we have now to perform $N(N+2Q+4)$ ones. 
With these simplifications, in the case when $\Vb_0$ and $\Vb$ are null matrices,
the global complexity of the algorithm is equal to $N^2(Q+2)/2$. When $Q=1$, we thus recover the
order of complexity of the classical RLS algorithm. Since the objective function then reduces to a quadratic function, 
Sherman-Morrison-Woodbury formula can be invoked to compute iteratively the minimizer on the whole space in an efficient manner.

Note finally that the computation of $\Xv_n \Xv_n^\top$ with $\correcEC{n\in \eN \setminus \left\{0\right\}} $, which needs to be performed
in Step \ref{step5}, remains a main source of complexity. However, if $(\forall n > Q)$
$\Xv_n = [\xv_{n-Q+1},\ldots,\xv_n]$ where $\xv_n\in \eR^N$ (as it is the case in affine projection based algorithms for adaptive processing \cite{Gay_S-L_1995_picassp_fast_apa}),
then a recursive computation of $\Xv_n \Xv_n^\top$ only requires $\xv_n \xv_n^\top$ to be computed at each iteration $n>Q$.
If we further assume that the model is a one-dimensional convolutive
one, i.e. $\xv_n$ corresponds to shifted samples of a signal $\big(x(n)\big)_{n\ge 1}$, then $(\forall n > N)$ $\xv_n = [x(n-N+1),\ldots,x(n)]^\top$
and $\xv_n \xv_n^\top$  can be itself computed recursively with a complexity of $N$ operations. Such ideas have been deeply investigated in the literature on fast RLS algorithms
\cite{Manolakis_D_2005_book_sasp}.

\section{Convergence study}
\label{sec:convergence}

\correcEC{Establishing the convergence of stochastic approximation algorithms is challenging \cite{Robbins1951,Kushner03,Frikha12,Fathi13,Combettes_P_2016_pafa_2016_Stochastic_apfbsmo}.}
Throughout this section and the related appendices, it is assumed that $\vartheta = 1$. 
The underlying probability space being denoted by $(\Omega,\FF,\PP)$, 
we will say in short that a property is $\as$ satisfied if this property holds almost surely.
\subsection{Assumptions}
For every \correcEC{$n\in \eN \setminus \left\{0\right\}$} , let $\XX_n = \sigma\big((\Xv_k,\yv_k)_{1\le k \le n}\big)$
be the sub-sigma algebra of $\FF$ generated by $(\Xv_k,\yv_k)_{1\le k \le n}$.
In order to give a proof of convergence of the proposed stochastic MM subspace algorithm,
we will make the following additional assumption:
\begin{assumption}\label{a:conv}\
\begin{enumerate}
\item \label{a:basici} $\Rb+\Vb_0$ is a positive definite matrix.
\item \label{a:convi} $\big((\Xv_n,\yv_n)\big)_{n\ge 1}$ is a stationary ergodic sequence and, for every \correcEC{$n\in \eN \setminus \left\{0\right\}$},
the elements of $\Xv_n$ and the components of $\yv_n$ have finite fourth-order moments.
\item \label{a:convibis} For every \correcEC{$n\in \eN \setminus \left\{0\right\}$},
\begin{align}
& \EC{\|\yv_{n+1}\|^2}{\XX_n} = \varrho\\
& \EC{\Xv_{n+1} \yv_{n+1}}{\XX_n} = \rb\\
& \EC{\Xv_{n+1}\Xv_{n+1}^\top}{\XX_n} = \Rb.
\end{align}
\item \label{a:conviii} $\hv_1$ is $\XX_1$-measurable and, for every \correcEC{$n\in \eN \setminus \left\{0\right\}$} , $\Dv_n$ is $\XX_n$-measurable.
\end{enumerate}
\end{assumption}
The following asymptotic results will then be useful in the rest of our developments.
 \begin{lemma}\label{a:asympt}
Under Assumptions \ref{a:conv}\ref{a:convi} and \ref{a:conv}\ref{a:convibis}, the following properties hold:
\begin{enumerate}
\item \label{a:asympti} $(\rho_n)_{n\ge 1}$, $(\Rv_n)_{n\ge 1}$, and $(\rv_n)_{n\ge 1}$ converge $\as$ to 
$\varrho$, $\Rb$ and $\rb$, respectively
\item \label{a:asymptii} $\displaystyle\sum_{n=1}^{+\infty} n^{-1} |\rho_n -\varrho| < +\infty$ $\;\;\;\as$\\
$\displaystyle\sum_{n=1}^{+\infty} n^{-1} \|\rv_n -\rb\| < +\infty$ $\;\;\as$\\
$\displaystyle\sum_{n=1}^{+\infty} n^{-1} |||\Rv_n -\Rb ||| < +\infty$ $\as$,
\end{enumerate}
where $|||\cdot|||$ denotes the spectral matrix norm.
\end{lemma}
\begin{proof} See Appendix \ref{ap:asympt}.
\end{proof}

\begin{remark}\
\begin{enumerate}
 \item Assumptions \ref{a:conv}\ref{a:convi} and \ref{a:conv}\ref{a:convibis} are more general than assuming
 that $\big((\Xv_n,\yv_n)\big)_{n\ge 1}$ is an independent identically distributed (i.i.d.) sequence and, for every \correcEC{$n\in \eN \setminus \left\{0\right\}$} ,
the elements of $\Xv_n$ and the components of $\yv_n$ have finite fourth-order moments.
 \item \correcEC{Assumption \ref{a:conv}\ref{a:conviii} is satisfied as soon as $\hv_1$ is $\XX_1$-measurable (e.g. $\hv_1$ is deterministic)  and the subspace directions, i.e., the columns of $\Dv_n$, only depend on $\big((\Xv_k,\yv_k,\hv_k)\big)_{1 \leq k \leq n}$. This is actually the case for the various subspace constructions listed in~\cite[Tab. I]{Chouzenoux2011a}, and, in particular, for the memory gradient subspace given by~\eqref{e:2memorysub}.}
\end{enumerate}
\end{remark}

\subsection{Almost sure convergence}

Let us give the following preliminary property:
\begin{lemma}\label{le:hbounded}
Under Assumptions \ref{a:basic}, \correcEC{\ref{ass:subspace} and} \ref{a:conv}\ref{a:convi}-\ref{a:conv}\ref{a:convibis},  $(\hv_n)_{n\ge 1}$ is $\as$ bounded.\footnote{We say that a sequence of random vectors is almost surely bounded when the norms of all these vectors can be bounded by some random variable with probability 1.}
\end{lemma}
\begin{proof}
See Appendix \ref{ap:hbounded}.
\end{proof}
\medskip
\noindent Combining the previous lemma with classical results on the asymptotic behaviour of almost supermartingales, the convergence of the sequence
$\big(\mathrm{F}_n(\hv_n)\big)_{n\ge 1}$ can be established:
\begin{lemma}\label{le:summahvn}
Under Assumptions \ref{a:basic}-\ref{a:conv},  $\big(\mathrm{F}_n(\hv_n)\big)_{n\ge 1}$ is $\as$ convergent
and $\big((\hv_{n+1}-\hv_n)^\top \Av_n(\hv_n) (\hv_{n+1}-\hv_n)\big)_{n\ge 1}$
is $\as$ summable.
\end{lemma}
\begin{proof} See Appendix \ref{ap:summahvn}.
\end{proof}
\medskip
\noindent  Lemma \ref{le:summahvn} allows us to deduce the following 
result on the sequence of gradients computed at each iteration of the algorithm:
\begin{lemma}\label{le:summagrad}
Under Assumptions \ref{a:basic}-\ref{a:conv}, 
$(\|\nabla \mathrm{F}_n(\hv_n)\|)_{n\ge 1}$ is $\as$ square-summable.
\end{lemma}
\begin{proof} See Appendix \ref{ap:summagrad}.
\end{proof}
\medskip 
\noindent By gathering all the previous results, our main convergence results can now be stated:
\begin{proposition}\label{p:convmain}
Assume that Assumptions \ref{a:basic}-\ref{a:conv} hold. Then, the following hold:
\begin{enumerate}
\item \label{p:convmaini} The set of cluster points of $(\hv_n)_{n\ge 1}$ is almost surely a nonempty compact connected set.
\item \label{p:convmainii} Any element of this set is almost surely a critical point of $F$.
\item \label{p:convmainiii} If the functions $(\psi_s)_{1\le s \le S}$ are convex, then $(\hv_n)_{n\ge 1}$ converges $\as$ to the unique (global)
minimizer of $F$.
\end{enumerate}
\end{proposition}
\begin{proof}
See Appendix \ref{ap:convmaini}.
\end{proof}
\medskip
\noindent It can be noticed that the conclusion of Proposition \ref{p:convmain}\ref{p:convmainiii} is still valid if the functions $(\psi_s)_{1\le s \le S}$
are nonconvex, they are twice continuously differentiable, and the regularization constants $(\lambda_s)_{1\le s \le S}$ as defined in Table \ref{t:penalty} are small enough
so that the function $F$ is strongly convex.

\subsection{\correcEC{Convergence rate}}

\correcEC{Based on our recent results in \cite{Chouzenoux16spl}, we provide a convergence rate result for Algorithm~\eqref{e:MMmini}
in the case when the functions $(\psi_s)_{1\le s \le S}$ are convex and twice differentiable. 
\begin{proposition} 
Suppose that Assumptions \ref{a:basic}-\ref{a:conv} hold. Let $\epsilon \in ]0,+\infty[$ be such that $\epsilon \Ib_N \prec \Rb+ \Vb_0 $. 
Then, there exists almost surely $n_\epsilon\in \eN \setminus \left\{0\right\}$ such that, for every $n\ge n_\epsilon$, $\nabla^2 \mathrm{F}_n(\hv_n) \succeq \Rb-\epsilon \Ib_N + \Vb_0$ and
\begin{equation}\label{e:boundinfFvnfinal}
\mathrm{F}_n(\hv_{n+1})-\inf \mathrm{F}_n \le \theta
\big(\mathrm{F}_n(\hv_n) - \inf \mathrm{F}_n\big)
\end{equation}
where $\theta \in [0,1)$.
\end{proposition}
More details about the expression of the decay rate can be found in \cite{Chouzenoux16spl}.
}

\section{Application to 2D \correcEC{system} identification} \label{se:ident}

\subsection{Problem statement}

We first demonstrate the efficiency of the proposed stochastic algorithm in a 2D \correcEC{system} identification problem. 
We consider the following observation model:
\begin{equation}
\yb = S(\overline{\hb}) \xb + \wb,
\end{equation}
where $\xb \in \eR^L$ and $\yb \in \eR^L$ represent the original and degraded versions of a given image, $\overline{\hb} \in \eR^{N}$ is the vectorized version of an unknown two-dimensional blur kernel, $S$ is the linear operator which maps the kernel to its associated Hankel-block Hankel matrix form, and $\wb \in \eR^L$ represents a realization of an additive noise. When the images $\xb$
and $\yb$ are of very large size, finding an estimate $\widehat{\hb} \in \eR^{N}$ of the blur kernel can be quite memory consuming, but one can expect good estimation performance by learning the blur kernel through a sweep of blocks in the dataset. 

Let us denote by $\Xb \in \eR^{L \times N}$ the matrix such that $S(\hb) \xb = \Xb \hb$. Then, we propose to define $\widehat{\hb}$ as a solution to \eqref{e:probdet}, where, for 
every \correcEC{$n\in \eN \setminus \left\{0\right\}$}, $\yv_n \in \eR^Q $ and $\Xv_n^\top \in \eR^{Q \times N}$, are subparts of $\yb$ and $\Xb$, respectively, corresponding to $Q\in \left\{1,\ldots,L\right\}$ lines of this vector/matrix.
For the regularization term $\Psi$, we consider, for every $s \in \left\{1,\ldots,N\right\}$ ($S=N$), an isotropic penalization on the gradient between neighboring coefficients of the blur kernel, i.e., $P_s = 2$ and $\Vb_s = \left[\Deltab_s^{\rm h} \, \, \Deltab_s^{\rm v} \right]^\top$, where $\Deltab_s^{\rm h}  \in \eR^N$ (resp. $\Deltab_s^{\rm v} \in \eR^N$) is the horizontal (resp. vertical) gradient operator applied at pixel $s$. The smoothness of $\hb$ is then enforced by choosing, for every $s \in \left\{1,\ldots,S\right\}$ 
and $u\in \eR$, $\psi_s(u) = \lambda \sqrt{1 + u^2/\delta^2}$ with $(\lambda, \delta)\in (0,+\infty)^2$. Finally, in order to guarantee the existence of a unique minimizer, the strong convexity of $F$ is imposed by taking $\vb_0 = \zerob$ and $\Vb_0 = \tau \Ib_N$, where $\tau$ is a small positive value (typically $\tau = 10^{-10}$). 

\subsection{Simulation results}

The original image, presented in Figure~\ref{fig:simul}(a), is a satellite image, of size $4096 \times 4096$ pixels. The original blur kernel $\overline{\hb}$ with size $21 \times 21$, and the resulting blurred image, which has been corrupted with a zero-mean white Gaussan noise with standard deviation $\sigma = 0.03$ (the blurred signal-to-noise ratio equals 25.7 dB), are displayed in Figures~\ref{fig:simul}(b)(c). Figure~\ref{fig:simul}(d)~presents the estimated kernel, using Algorithm~\ref{e:stochMMalg}, with the subspace given by \eqref{e:2memorysub}, leading to the so-called \correcEC{stochastic MM memory gradient (S3MG)} algorithm. Parameters $(\lambda, \delta)$ were adjusted so as to minimize the normalized root mean square estimation error, here equal to $0.064$. Figure~\ref{fig:comp} illustrates the variations of this estimation error with respect to the computation time for the proposed algorithm, the SGD algorithm with a decreasing stepsize proportional to $n^{-1/2}$, the regularized dual averaging (RDA) method with a constant stepsize from \cite{Xiao2010}, and the accelerated stochastic gradient averaging \correcEC{SAGA method with a constant stepsize from \cite{Defazio14b}. Tests were running } on an Intel(R) Xeon(R) E5-2630 @ 2.6GHz using a Matlab~7 implementation. Note that for the latter three algorithms, the stepsize parameter was optimized manually so as to obtain the best performance in terms of convergence speed. Finally, note that all tested algorithms were observed to provide asymptotically the same estimation quality, whatever the size of the blocks. In this example, as illustrated in Figure~\ref{fig:blk}, the best trade-off in terms of convergence speed is obtained for $Q = 256 \times 256$. 

\begin{figure}[h!]
\centering
\begin{tabular}{@{}c@{}c@{}}
\includegraphics[width=4.5cm]{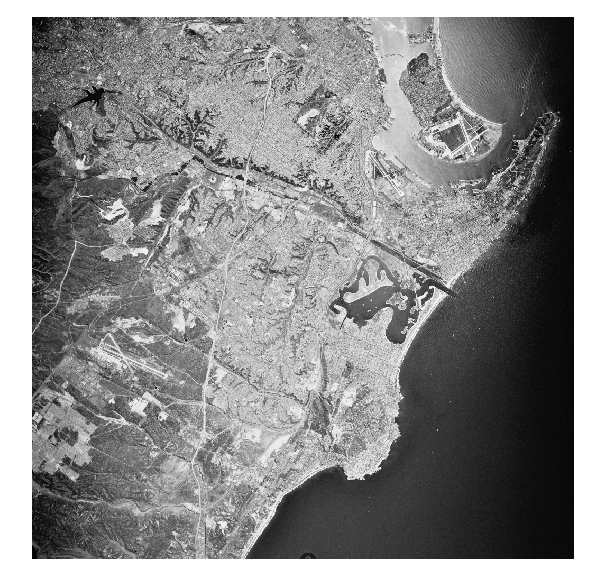} & \includegraphics[width=4.5cm]{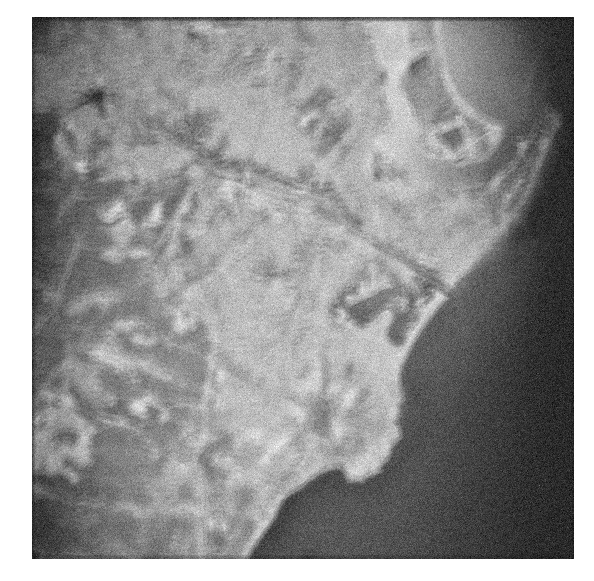} \\
(a) & (b)\\
\includegraphics[width=4.5cm]{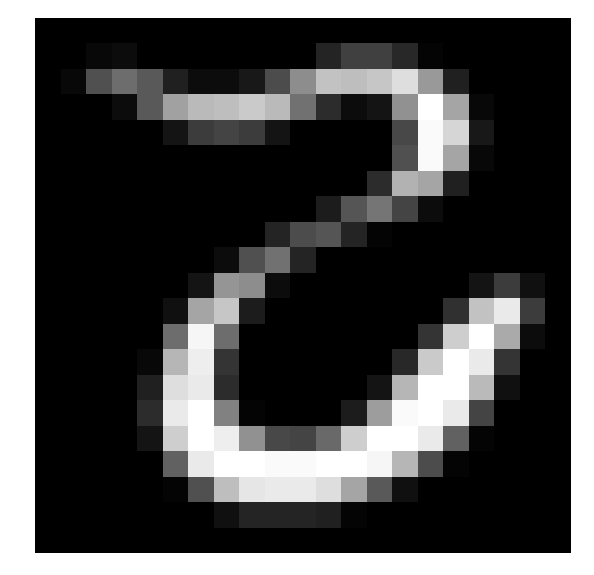} & \includegraphics[width=4.5cm]{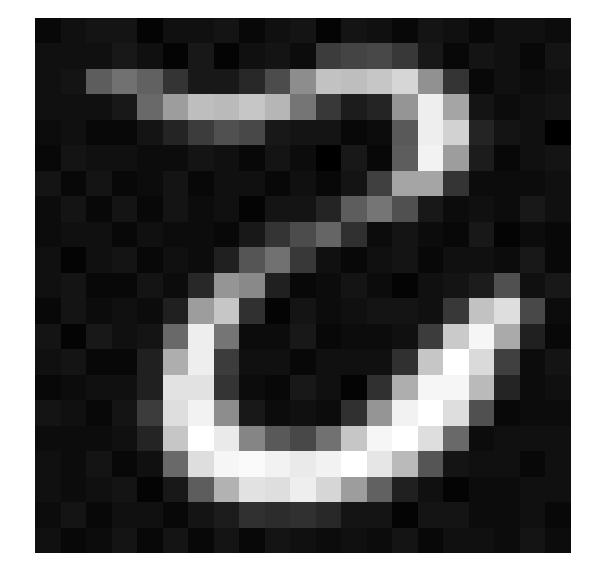}\\
(c) & (d)
\end{tabular}
\caption{(a) Original image. (b) Blurred and noisy image. (c) Original blur kernel. (d) Estimated blur kernel, with relative error $0.064$.}
\label{fig:simul}
\end{figure}

\begin{figure}[h]
\centering
\includegraphics[width=8cm]{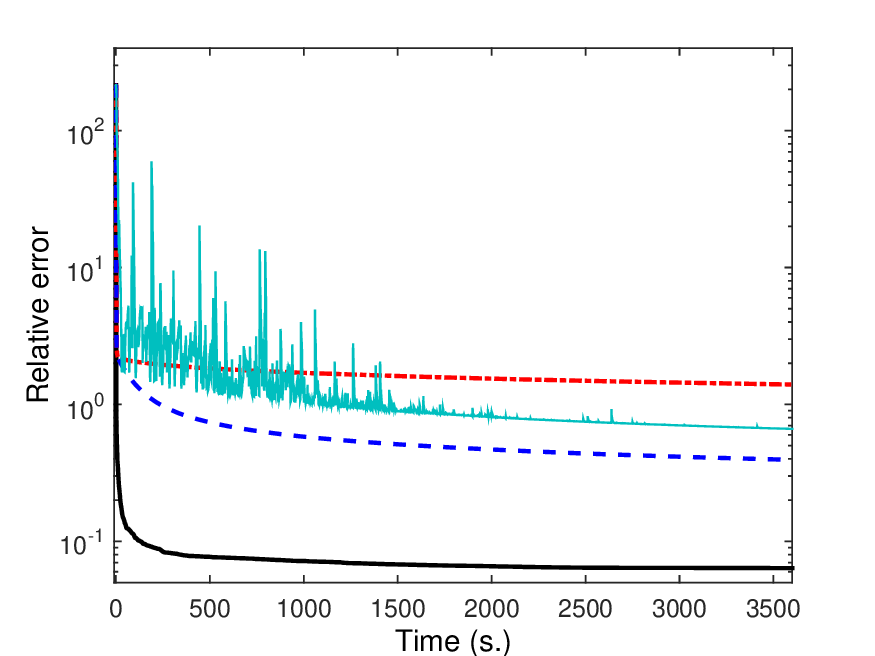}
\caption{Comparison of S3MG algorithm (solid black line),
SGD algorithm with decreasing stepsize $\propto n^{-1/2}$ (dashed-dotted red line), RDA algorithm with constant stepsize (dashed blue line) and \correcEC{SAGA} algorithm with constant stepsize (turquoise thin line).}
\label{fig:comp}
\end{figure}

\begin{figure}[h]
\centering
\includegraphics[width=8cm]{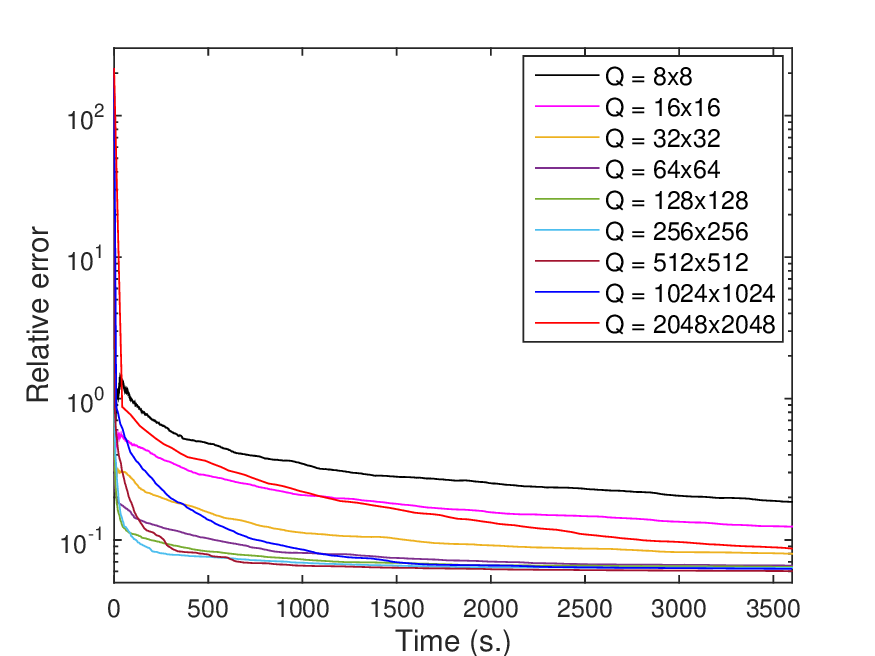}
\caption{Effect of the block size $Q$ on the convergence speed of S3MG.}
\label{fig:blk}
\end{figure}


\section{Application to sparse adaptive filtering}
\label{sec:adapt}
\subsection{Problem statement}
As emphasized in Sections \ref{se:probform} and \ref{se:propmeth}, one of the advantages of Algorithm \ref{e:stochMMalg} compared with some other online optimization algorithms is that it is able to
deal with adaptive data processing problems.
In this section, we apply the S3MG algorithm to the identification of a sparse time-varying system. Given a real-valued  discrete-time input signal $\big(x(n)\big)_{n\in \eZ}$, the output of the system
at time $n\ge 1$ is defined as
\begin{equation}\label{e:modadapt}
y_n = \Xb_n^\top \overline{\hb}_n + w_n,
\end{equation}
where $\Xb_n = [x(n-N+1),\ldots,x(n)]^\top$, $w_n$ models some measurement noise, and $\overline{\hb}_n\in \eR^N$ gathers the unknown filter taps at time $n$.
Then, the objective is to provide an estimate of the vector $\overline{\hb}_n$ at each time by
solving Problem~\eqref{e:probdet} where the regularization function $\Psi$ is chosen in order to promote the sparsity of the impulse response of the time-varying filter.

\subsection{Simulation results}
We generate data according to Model \eqref{e:modadapt} where the input signal $\big(x(n)\big)_{n\in \eZ}$ consists of identically and independent random binary values $\left\{-1,+1\right\}$. The
measurement noise $(w_n)_{n\in \eZ}$ is white Gaussian with zero mean and variance $0.05$. In order to evaluate the tracking capability of the proposed S3MG method, 
the following time-varying linear system is considered:
 \begin{equation}
 \overline{\hb}_n = 
 \begin{cases}
 \overline{\hb}_1 & \mbox{if $n \leq L/2$,}\\
 \overline{\hb}_{L/2+1}  & \mbox{if $n \geq L/2+1$.}
 \end{cases}
 \end{equation}
 The filter length $N$ is equal to 200 and the output of the system is observed at every time $n\in \{1,\ldots,L\}$ with $L = 5000$.
The sparse impulse responses corresponding to vectors $ \overline{\hb}_1$ and $\overline{\hb}_{L/2+1} $ are represented in Figure~\ref{fig:filters}.

We compute, for every $n \in \left\{1,\ldots,L\right\}$, the Euclidean norm of the error between the current estimate $\hb_n$ and the true filter coefficient vector $\overline{\hb}_n$. The minimal estimation error is obtained for the nonconvex Welsch penalty function (see Table \ref{t:penalty}) and 
a smoothed $\ell_2-\ell_0$ regularization function is thus employed by setting $S = N$, $\vb_0 = \zerob$, $\Vb_0 =  \Ob_N$, 
and, for every $s\in \{1,\ldots,N\}$, $P_s= 1$, $\vb_s = 0$, while $\Vb_s \in \eR^{1\times N}$ is the $s$-th vector of the canonical basis of $\eR^N$.

We present the results generated by S3MG in Figure~\ref{fig:resfilter2} for two values of the forgetting factor $\vartheta$, namely $\vartheta = 1$ which corresponds to a non adaptive strategy, and $\vartheta = 0.995$ which appears to be the best choice in terms of tracking properties for this example. 

We also show the results obtained with several state-of-the-art approaches in the context of sparse adaptive filtering, namely SPAL \cite{Kopsinis_Y_2011_ieee-tsp_online_ssisr}, RLMS \cite{Chen_Y_2010_unpublished_Regularized_LMSa}, RZAAPA \cite{Meng_R_2011_procssdp_sparsity_aapaasi} and SM-PAPA \cite{Werner07}. Note that, for each tested method, the involved parameters (stepsize, regularization weight, blocksize) have been tuned manually in order to optimize the performance in terms of error decay.

\begin{figure}[h]
\centering
\includegraphics[width=8cm]{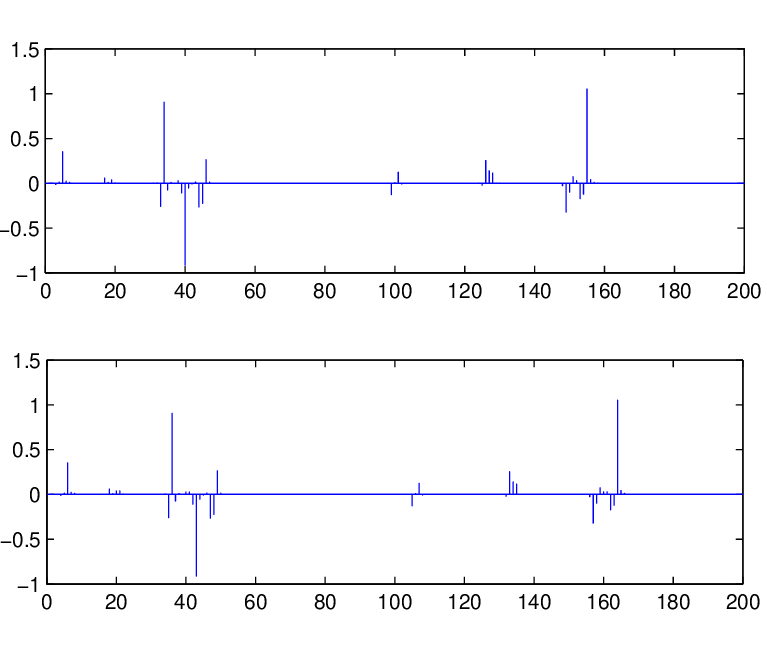}
\caption{Values of the coefficients of the considered sparse filters $\overline{\hb}_1$ (top) and $\overline{\hb}_{L/2+1}$ (bottom).}
\label{fig:filters}
\end{figure}


\begin{figure}[h]
\centering
\includegraphics[width=8cm]{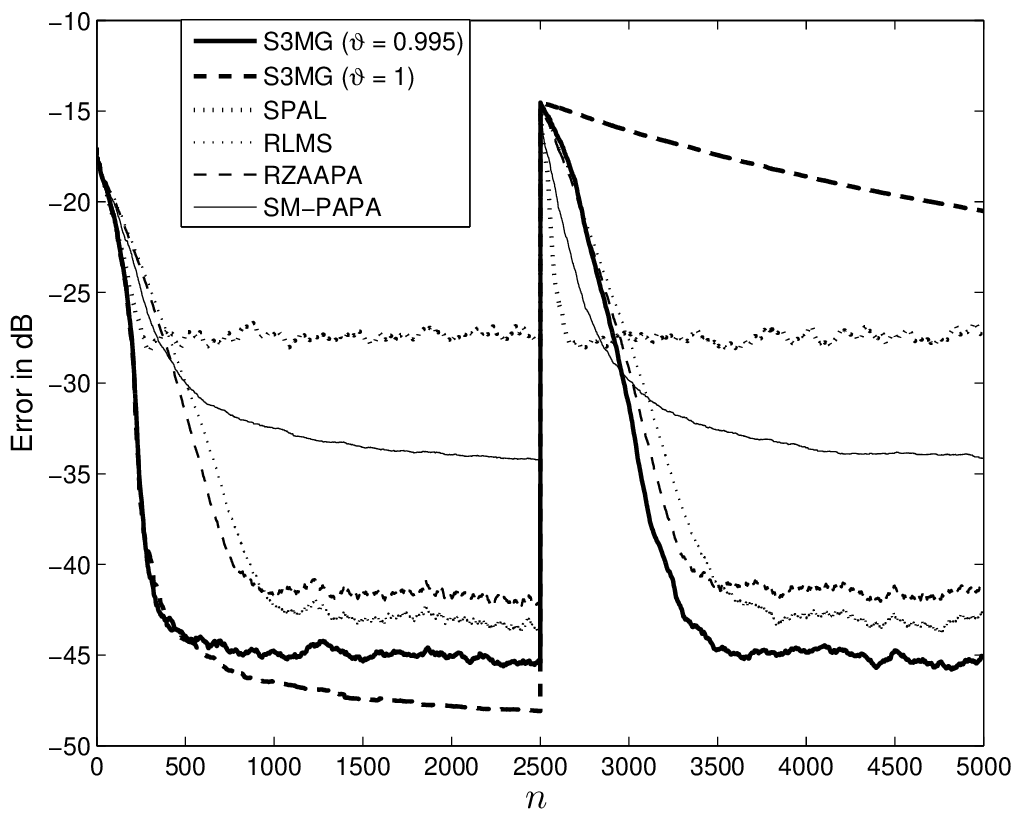}
\caption{Quadratic estimation error on the filter coefficients as a function of time index $n$ for various adaptive algorithms.}
\label{fig:resfilter2}
\end{figure}

\section{Conclusion} \label{se:conclu}
In this work, we have proposed a stochastic MM subspace algorithm for online penalized least squares estimation problems.
The method makes it possible to use large-size datasets the second-order moments of which are not known a priori. We have shown 
that the proposed algorithm is of the same order of complexity as the classical RLS algorithm and that its computational cost 
can be reduced by taking advantage of specific forms of the search subspace. The choice of a memory gradient subspace
led to the S3MG algorithm whose good numerical performance has been demonstrated in the context of 2D \correcEC{system} identification for large scale image processing problems. 
In the context of sparse adaptive filtering, S3MG has also been shown to be competitive with respect to recent methods. 
Although an analysis of the convergence of the proposed method has been carried out,
it would be interesting to extend the obtained results to weaker assumptions. In addition, in a nonstationary context, a theoretical
study of the tracking abilities of the algorithm should be conducted. \correcEC{Finally, let us emphasize that a detailed analysis of the convergence rate of the proposed method has been undertaken in our recent paper~\cite{Chouzenoux16spl}.}

\appendix

\section{Proof of Lemma \ref{a:asympt}} \label{ap:asympt}
Property \ref{a:asympti} is a consequence of the ergodic theorem \cite[Theorem 13.12]{Davidson_J_1994_book_Stochastic_lt}.
In addition, the law of the iterated logarithm for martingale difference sequences \cite{Stout_W_1970-ann_math_stat_Hartman_wlil} ensures that
\begin{align}
&\limsup_{n\to +\infty} \frac{|\sum_{k=1}^n (\|\yv_k\|^2 - \varrho)|}{\big(n \log(\log n)\big)^{1/2}} < +\infty \qquad \;\as\\
&\limsup_{n\to +\infty} \frac{\|\sum_{k=1}^n (\Xv_k \yv_k - \rb)\|}{\big(n \log(\log n)\big)^{1/2}} < +\infty \qquad \as\\
&\limsup_{n\to +\infty} \frac{|||\sum_{k=1}^n (\Xv_k \Xv_k^\top - \Rb)|||}{\big(n \log(\log n)\big)^{1/2}} < +\infty
 \qquad \!\!\!\as
\end{align}
that is
\begin{align}
&\limsup_{n\to +\infty} \frac{n^{1/2}|\rho_n - \varrho|}{\big(\log(\log n)\big)^{1/2}} < +\infty  \qquad \as \label{e:rhonlogiterbis}\\
&\limsup_{n\to +\infty} \frac{n^{1/2}\|\rv_n - \rb\|}{\big(\log(\log n)\big)^{1/2}} < +\infty  \qquad \as \label{e:rvnlogiterbis}\\\
&\limsup_{n\to +\infty} \frac{n^{1/2}|||\Rv_n - \Rb |||}{\big(\log(\log n)\big)^{1/2}} < +\infty  \qquad \!\as \label{e:Rvnlogiterbis}
\end{align}
Consequently, for every $n_0\in\eN$ with $n_0\ge 2$,
\begin{equation}
\sum_{n=n_0}^{+\infty} n^{-1} |\rho_n -\varrho| \le \sup_{n\ge n_0} \left(\frac{n^{1/2}|\rho_n - \varrho|}{\big(\log(\log n)\big)^{1/2}}\right)
\Big(\sum_{n=n_0}^{+\infty} n^{-3/2} |\log(\log n)|^{1/2}\Big).
\end{equation}
Since $\sum_{n=2}^{+\infty} n^{-3/2} |\log(\log n)|^{1/2} < +\infty$, it follows from \eqref{e:rhonlogiterbis} that 
$\sum_{n=n_0}^{+\infty} n^{-1} |\rho_n -\varrho|$ converges $\as$ to 0 as $n_0 \to +\infty$, which means that the first line
in Property \ref{a:asymptii} is satisfied. By proceeding similarly, \eqref{e:rvnlogiterbis}
and \eqref{e:Rvnlogiterbis} allow us to establish the remaining two assertions in Property \ref{a:asymptii}.

\section{Proof of Lemma \ref{le:hbounded}} \label{ap:hbounded}
For every \correcEC{$n\in \eN \setminus \left\{0\right\}$}, minimizing $\Theta_n(\cdot,\hv_n)$ is equivalent to minimizing the function
\begin{equation}\label{e:deftildeTheta}
(\forall \hb \in \eR^N) \quad 
\widetilde{\Theta}_n(\hb,\hv_n) = \frac12 \hb^\top \Av_n(\hv_n) \hb-\cv_n(\hv_n)^\top \hb.
\end{equation}
It follows from Assumption \ref{a:conv}\ref{a:convi}-\ref{a:conv}\ref{a:convibis} and Lemma \ref{a:asympt}\ref{a:asympti} that there exists 
$\Lambda\in\FF$ such that $\PP(\Lambda)=1$ and, for every $\omega\in \Lambda$, 
\begin{align}
&\lim_{n\to +\infty} \rv_n(\omega) = \rb \label{e:convrnas}\\
&\lim_{n\to +\infty} \Rv_n(\omega) = \Rb \label{e:convRnas}.
\end{align}
Let $\omega \in \Lambda$. According to Assumption \ref{a:basic}\ref{a:basiciv} and Eq. \eqref{e:defbb}, 
$\bb(\hb)$ is bounded as a function of $\hb$. It is then deduced from \eqref{e:defcv} and \eqref{e:convrnas} that
$\big(\cv_n(\hv_n)(\omega)\big)_{n\ge 1}$ is bounded, i.e. there exists $\eta \in [0,+\infty)$
such that
\begin{equation}\label{e:boundc}
(\forall n \in \eN \setminus \left\{0\right\})\qquad \|\cv_n(\hv_n)(\omega)\| \le \eta.
\end{equation}
In addition, as a consequence of \eqref{e:defbb} and Assumption \ref{a:basic}\ref{a:basiciv}, for every \correcEC{$n\in \eN \setminus \left\{0\right\}$},
$\Diag\big(\bb(\hv_n)\big)$ is a positive semidefinite matrix.
Hence,  because of \eqref{e:defA}, Assumptions \ref{a:basic}\ref{a:basiciv} and \ref{a:conv}\ref{a:basici}, 
and \eqref{e:convRnas}, there exists $\epsilon \in (0,+\infty)$
and $n_0 \in \eN \setminus \left\{0\right\}$
such that 
\begin{equation}\label{e:boundA}
(\forall n \ge n_0)\qquad
\Av_n(\hv_n)(\omega) \succeq \Rb-\epsilon \Ib_N+\Vb_0 \succ \Ob_N.
\end{equation}
(It suffices to choose $\epsilon$ lower than the minimum eigenvalue of $\Rb+\Vb_0$).
As a consequence of \eqref{e:deftildeTheta}, \eqref{e:boundc}, \eqref{e:boundA}, and the Cauchy-Schwarz inequality, we have
\begin{equation}
(\forall n \ge n_0)
(\forall \hb \in \eR^N) \quad 
\frac12 \hb^\top(\Rb-\epsilon \Ib_N+\Vb_0)\hb- \eta \|\hb\| \le \widetilde{\Theta}_n(\hb,\hv_n).
\end{equation}
Since $\Rb-\epsilon \Ib_N+\Vb_0$ is a positive definite matrix, 
the lower bound corresponds to a coercive function with respect to $\hb$. There thus exists
$\zeta \in (0,+\infty)$ such that, for every $\hb \in \eR^N$,
\begin{equation}
\|\hb\| > \zeta\quad \Rightarrow \qquad (\forall n \ge n_0)\;\; \widetilde{\Theta}_n(\hb,\hv_n)(\omega) > 0.
\end{equation}
On the other hand, since $\zerob \in \Span\big(\Dv_n(\omega)\big)$, we have
\begin{equation}
\widetilde{\Theta}_n(\hv_{n+1},\hv_n)(\omega) \le \widetilde{\Theta}_n(\zerob,\hv_n)(\omega) = 0.
\end{equation}
The last two inequalities allow us to conclude that
\begin{equation}
(\forall n \ge n_0)\qquad 
\|\hv_{n+1}(\omega)\| \le \zeta.
\end{equation}

\section{Proof of Lemma \ref{le:summahvn}} \label{ap:summahvn}
According to Assumption \correcEC{\ref{ass:subspace}}, the proposed algorithm is actually equivalent to
\begin{align}
(\forall n \in \eN \setminus \left\{0\right\}) \qquad &\hv_{n+1} = \hv_n + \Dv_n \widetilde{\uv}_n \label{e:uphvnuvt}\\
& \widetilde{\uv}_n = \argmin_{\widetilde{\ub} \in \eR^M} \Theta_n(\hv_n + \Dv_n \widetilde{\ub},\hv_{n}).
\end{align}
By using \eqref{e:defQc} and cancelling the derivative of the function $\widetilde{\ub} \mapsto \Theta_n(\hv_n + \Dv_n \widetilde{\ub},\hv_{n})$,
\begin{equation}\label{e:opttildeuvn}
\Dv_n^\top \nabla \mathrm{F}_n(\hv_n)  + \Dv_n^\top \Av_n(\hv_n) \Dv_n \widetilde{\uv}_n = \zerob.
\end{equation}
Hence,
\begin{align}\label{e:majFnhvnuvt}
& \Theta(\hv_{n+1},\hv_n) \nonumber\\
 = &\;\mathrm{F}_n(\hv_{n})-\frac12 \widetilde{\uv}_n^\top \Dv_n^\top \Av_n(\hv_n) \Dv_n \widetilde{\uv}_n\nonumber\\
 = &\;\mathrm{F}_n(\hv_{n})-\frac12 (\hv_{n+1}-\hv_n)^\top \Av_n(\hv_n)  (\hv_{n+1}-\hv_n).
\end{align}
In view of \eqref{e:MMmaj} and Proposition~\ref{p:majquad}, this yields
\begin{equation}\label{e:majFnhvnp1hvn}
(\forall n \in \eN \setminus \left\{0\right\})\;\;
\mathrm{F}_n(\hv_{n+1})+ \frac12 (\hv_{n+1}-\hv_n)^\top \Av_n(\hv_n)  (\hv_{n+1}-\hv_n)\\ \le \mathrm{F}_n(\hv_{n}).
\end{equation}
In addition, the following recursive relation holds
\begin{align}
(\forall \hb \in \eR^N)\quad
&\mathrm{F}_{n+1}(\hb)= \mathrm{F}_n(\hb)+ \frac12 (\rho_{n+1}-\rho_n)\nonumber\\ 
&- (\rv_{n+1}-\rv_n)^\top \hb+ \frac12 \hb^\top (\Rv_{n+1}-\Rv_n) \hb.
\end{align}
As a consequence of Assumption \ref{a:conv}\ref{a:conviii}, for every \correcEC{$n\in \eN \setminus \left\{0\right\}$}, $\hv_{n+1}$ is $\XX_n$-measurable.
It can thus be deduced from \eqref{e:majFnhvnp1hvn} and the previous two relations that
\begin{equation}\label{e:recurFnrand}
\EC{\mathrm{F}_{n+1}(\hv_{n+1})}{\XX_n} + \frac12 (\hv_{n+1}-\hv_n)^\top 
\Av_n(\hv_n) (\hv_{n+1}-\hv_n)
 \le \;\mathrm{F}_n(\hv_{n})+ \chi_n
\end{equation}
where
\begin{equation}
\chi_n= \frac12 \EC{\rho_{n}-\rho_{n+1}}{\XX_n}- \EC{\rv_n-\rv_{n+1}}{\XX_n}^\top \hv_{n+1}
+ \frac12 \hv_{n+1}^\top \EC{\Rv_n-\Rv_{n+1}}{\XX_n} \hv_{n+1}.
\end{equation}
By using \eqref{e:defsamprho}-\eqref{e:defsampR} with $\vartheta = 1$ and Assumption \ref{a:conv}\ref{a:convibis}, we have
\begin{align}
\chi_n = & \;\frac{1}{2(n+1)} \big(\rho_n-\EC{\|\yv_{n+1}\|^2}{\XX_n}\big)\nonumber\\ 
&- 
\frac{1}{n+1} \big(\rv_n-\EC{\Xv_{n+1}\yv_{n+1}}{\XX_n}\big)^\top \hv_{n+1}\nonumber\\
&+ \frac{1}{2(n+1)} \hv_{n+1}^\top \big(\Rv_n-\EC{\Xv_{n+1}\Xv_{n+1}^\top}{\XX_n}\big) \hv_{n+1}\nonumber\\
= & \;\frac{1}{2(n+1)} \big(\rho_n-\varrho\big) - 
\frac{1}{n+1} \big(\rv_n-\rb\big)^\top \hv_{n+1}\nonumber\\
&+ \frac{1}{2(n+1)} \hv_{n+1}^\top \big(\Rv_n-\Rb\big) \hv_{n+1}
\end{align}
which yields
\begin{align}
|\chi_n| \le & \;\frac{1}{2(n+1)} |\rho_n-\varrho| + 
\frac{1}{n+1} \|\rv_n-\rb\| \|\hv_{n+1}\|\nonumber\\
&+ \frac{1}{2(n+1)} |||\Rv_n-\Rb||| \,\|\hv_{n+1}\|^2.
\end{align}
According to Lemma \ref{le:hbounded}, $(\hv_n)_{n\ge 1}$ is $\as$ bounded, and Assumptions~\ref{a:conv}\ref{a:convi}-\ref{a:conv}\ref{a:convibis}
and Lemma~\ref{a:asympt}\ref{a:asymptii} thus guarantee that
\begin{equation}\label{e:sumchin}
\sum_{n=1}^{+\infty} |\chi_n| <+\infty \qquad \as
\end{equation}
Assumption \ref{a:basic}\ref{a:basicii} entails that, for every $n\in \eN \setminus \left\{0\right\}$, $\mathrm{F}_n$ is lower bounded by
$\inf \Psi > -\infty$. Furthermore, \eqref{e:recurFnrand} leads to
\begin{equation}
\EC{\mathrm{F}_{n+1}(\hv_{n+1})-\inf \Psi}{\XX_n} + \frac12 (\hv_{n+1}-\hv_n)^\top 
\Av_n(\hv_n) (\hv_{n+1}-\hv_n) \le \;\mathrm{F}_n(\hv_{n})-\inf \Psi+ |\chi_n|.
\end{equation}
Since, for every \correcEC{$n\in \eN \setminus \left\{0\right\}$}, $\mathrm{F}_n(\hv_{n})-\inf \Psi$ and $(\hv_{n+1}-\hv_n)^\top 
\Av_n(\hv_n) (\hv_{n+1}-\hv_n)$ are nonnegative, $(\mathrm{F}_n(\hv_{n})-\inf \Psi)_{n\ge 1}$ is
a nonnegative almost supermartingale \cite{Lai_T_2009_jelecthistprobstat_Martingales_sats}.
By invoking now Siegmund-Robbins lemma \cite{Robbins_H_1971_book_convergence_tfnnas}, it can be deduced from \eqref{e:sumchin} that 
the desired convergence results hold.

\section{Proof of Lemma \ref{le:summagrad}} \label{ap:summagrad}
 According to \eqref{e:defQc}, we have, for every $\phi \in \eR$ and \correcEC{$n\in \eN \setminus \left\{0\right\}$},
\begin{equation}
\Theta_n\big(\hv_n-\phi \nabla \mathrm{F}_n(\hv_n),\hv_n\big) = \mathrm{F}_n(\hv_n) - \phi \|\nabla \mathrm{F}_n(\hv_n)\|^2\\
+\frac{\phi^2}{2} \big(\nabla \mathrm{F}_n(\hv_n)\big)^\top \Av_n(\hv_n) \nabla \mathrm{F}_n(\hv_n).
\end{equation}
Let
\begin{equation}
 \Phi_n \in \Argmind{\phi \in \eR}{\Theta_n\big(\hv_n-\phi \nabla \mathrm{F}_n(\hv_n),\hv_n\big)}.
\end{equation}
The following optimality condition holds:
\begin{equation}\label{e:optPhin}
  \big(\nabla \mathrm{F}_n(\hv_n)\big)^\top \Av_n(\hv_n) \nabla \mathrm{F}_n(\hv_n)\, \Phi_n = \|\nabla \mathrm{F}_n(\hv_n)\|^2.
\end{equation}
As a consequence of Assumption~\ref{ass:subspace}, $(\forall \phi\in \eR)$ $\hv_n-\phi \nabla \mathrm{F}_n(\hv_n) \in \Span \Dv_n$.
It then follows from \eqref{e:MMmini} and  \eqref{e:optPhin} that 
\begin{align}\label{e:difFngradFnPhin}
\Theta_n\big(\hv_{n+1},\hv_n\big) &\le \Theta_n\big(\hv_n-\Phi_n \nabla \mathrm{F}_n(\hv_n),\hv_n\big)\nonumber\\
& \le \mathrm{F}_n(\hv_n) -\frac{\Phi_n}{2} \|\nabla \mathrm{F}_n(\hv_n)\|^2
\end{align}
which, by using \eqref{e:majFnhvnuvt}, leads to
\begin{equation}\label{e:majgradsq}
\Phi_n \|\nabla \mathrm{F}_n(\hv_n)\|^2 \le (\hv_{n+1}-\hv_n)^\top 
\Av_n(\hv_n) (\hv_{n+1}-\hv_n).
\end{equation}

Let $\epsilon > 0$.
Assumption \ref{a:basic}\ref{a:basiciv} and \eqref{e:defA} yield,
for every \correcEC{$n\in \eN \setminus \left\{0\right\}$} ,
\begin{equation}
\Av_n(\hv_n) \preceq  (||| \Rv_n+\Vb_0 |||+\overline{\nu} ||| \Vb |||^2)  \Ib_N.
\end{equation}
Therefore, according to Assumptions \ref{a:conv}\ref{a:basici} and \ref{a:conv}\ref{a:convi}, and Lemma~\ref{a:asympt}\ref{a:asympti}, there exists 
$\Lambda\in\FF$ such that $\PP(\Lambda)=1$ and, for every $\omega\in \Lambda$, 
\begin{equation}\label{e:boundAA}
(\exists n_0 \in \eN \setminus \left\{0\right\})
(\forall n \ge n_0)\quad
\Ob_N\prec \Av_n(\hv_n)(\omega) \preceq  \alpha_\epsilon^{-1} \Ib_N
\end{equation}
where
\begin{equation}\label{e:defalpha}
\alpha_\epsilon= (|||\Rb+\Vb_0|||+\overline{\nu} |||\Vb|||^2+\epsilon)^{-1} > 0.
\end{equation}
Let $\omega \in \Lambda$.
By using now \eqref{e:optPhin}, it can be deduced from \eqref{e:boundAA} that, if $n\ge n_0$
and $\nabla \mathrm{F}_n(\hv_n)(\omega) \neq \zerob$, then
\begin{equation}
 \Phi_n(\omega) \ge \alpha_\epsilon.
\end{equation}
Then, it follows from  \eqref{e:majgradsq} that
\begin{align}
& \alpha_\epsilon \sum_{n=n_0}^{+\infty}  \|\nabla \mathrm{F}_n(\hv_n)(\omega)\|^2 \nonumber\\
 &\leq 
 \sum_{n=n_0}^{+\infty} \big(\hv_{n+1}(\omega)-\hv_n(\omega)\big)^\top 
\Av_n(\hv_n)(\omega) \big(\hv_{n+1}(\omega)-\hv_n(\omega)\big).
\end{align}
By invoking Lemma  \ref{le:summahvn}, we can conclude that
$(\|\nabla \mathrm{F}_n(\hv_n)\|^2)_{n\ge 1}$ is $\as$ summable.

\section{Proof of Proposition \ref{p:convmain}} \label{ap:convmaini}
It follows from Lemma \ref{le:summahvn} that $\big((\hv_{n+1}-\hv_n)^\top \Av_n(\hv_n) (\hv_{n+1}-\hv_n)\big)_{n\ge 1}$
 converges $\as$ to 0. In addition, we have seen in the proof of Lemma \ref{le:hbounded} that
there exists 
$\Lambda\in\FF$ such that $\PP(\Lambda)=1$ and, for every $\omega\in \Lambda$, \eqref{e:boundA}
holds with $\epsilon \in (0,+\infty)$ and \correcEC{$n_0\in \eN \setminus \left\{0\right\}$}. This implies
that, for every $n \ge n_0$,
\begin{multline}
|||\Rb-\epsilon \Ib_N+\Vb_0|||\, \|\hv_{n+1}(\omega)-\hv_n(\omega)\|^2 \\
\;\;\le \big(\hv_{n+1}(\omega)-\hv_n(\omega)\big)^\top \Av_n(\hv_n)(\omega) 
\big(\hv_{n+1}(\omega)-\hv_n(\omega)\big)
\end{multline}
where $|||\Rb-\epsilon \Ib_N+\Vb_0||| > 0$. Consequently, $(\hv_{n+1}-\hv_n)_{n\ge 1}$ converges $\as$ to $\zerob$.
In addition, according to Lemma~\ref{le:hbounded}, $(\hv_n)_{n\ge 1}$ belongs almost surely to a compact set.
The result is then obtained by invoking Ostrowski's theorem \cite[Theorem 26.1]{Ostrowki_A_M_1973_book_Solution_eebs}.

\noindent\ref{p:convmainii} By using \eqref{e:gradFn}-\eqref{e:defcv}, we have
\begin{equation}
(\forall \correcEC{n\in \eN \setminus \left\{0\right\}} )\quad
\nabla \mathrm{F}_n(\hv_n)-\nabla F(\hv_n) = (\Rv_n-\Rb)\hv_n - \rv_n+\rb.
\end{equation}
Since $(\hv_n)_{n\ge 1}$ is almost surely bounded, it follows from Lemma \ref{a:asympt}\ref{a:asympti}
that $\big(\nabla\mathrm{F}_n(\hv_n)-\nabla F(\hv_n)\big)_{n\ge 1}$ converges $\as$ to $\zerob$.
Since Lemma \ref{le:summagrad} ensures that $\big(\nabla \mathrm{F}_n(\hv_n)\big)_{n\ge 1}$ converges $\as$ to $\zerob$,
$\big(\nabla F(\hv_n)\big)_{n\ge 1}$ also converges $\as$ to $\zerob$. There thus exists 
$\Lambda\in\FF$ such that $\PP(\Lambda)=1$ and, for every $\omega\in \Lambda$,
$\nabla F\big(\hv_n(\omega)\big)  \to \zerob$. Let $\widehat{\hb}$ be a cluster point of $\big(\hv_n(\omega)\big)_{n\ge 1}$.
There exists a subsequence $\big(\hv_{k_n}(\omega)\big)_{n\ge 1}$ such that $\hv_{k_n}(\omega) \to \widehat{\hb}$.
As we have assumed that the regularization functions $(\psi_s)_{1\le s \le S}$ are continuously differentiable (see Assumption \ref{a:basic}\ref{a:basicii}),
$F$ is also continuously differentiable, and 
\begin{equation} 
\nabla F(\widehat{\hb}) = \lim_{n\to +\infty} \nabla F\big(\hv_{k_n}(\omega)\big) = \zerob.
\end{equation}
This means that $\widehat{\hb}$ is a critical point of $F$.

\noindent\ref{p:convmainiii}  Because of Assumption \ref{a:conv}\ref{a:basici}, when the functions $(\psi_s)_{1\le s \le S}$
are convex, $F$ is a strongly convex function. It thus possesses a unique critical point $\widehat{\hb}$, which is the global minimizer
of $F$. It follows from \ref{p:convmaini} and \ref{p:convmainii} that, almost surely, the unique cluster point of $(\hv_n)_{n\ge 1}$ is $\widehat{\hb}$,
which shows that $\hv_n \to \widehat{\hb}$ $\as$

\section*{Acknowledgements} The authors would like to thank Professors Markus V. S. Lima and Paulo S. R. Diniz from Federal University of Rio de Janeiro for kindly making us accessible some of their codes. We would also like to thank Professor Anisia Florescu from Dun\u{a}rea de Jos University of Gala\c{t}i for providing the initial motivation for this work.

\bibliographystyle{IEEEbib}
\bibliography{stoch}

\end{document}